\documentclass[oneside,10pt]{article}          
\usepackage[b5paper]{geometry}	    
\usepackage{amsfonts,amsmath,latexsym,amssymb} 
\usepackage{amsthm}                
\usepackage{mathrsfs,upref}         
\usepackage{mathptmx}		    
\usepackage{graphicx}
\usepackage{color,soul}
\usepackage{float}
\usepackage{lineno}
\usepackage{pgf}
\usepackage{fdc}	            
\theoremstyle{definition}
\newtheorem{theorem}{Theorem}           
\newtheorem{lemma}{Lemma}

\newtheorem{remark}{Remark}

\begin{document}

\title[Short title]{Global stability of a Leslie-Gower-type fractional order tritrophic food chain model}


\author{Shuvojit Mondal, Nandadulal Bairagi and Gaston M. N'Guerekata}

\address{Shuvojit Mondal \\ Department of Mathematics\\ Rabindra Mahavidyalaya, Champadanga, Hooghly-712401\\ West Bengal, India\\
\email{shuvojitmondal91@gmail.com}}

\address{Nandadulal Bairagi \\ Centre for Mathematical Biology and Ecology\\ Department of Mathematics, Jadavpur University\\ Kolkata-700032, India\\
\email{nbairagi.math@jadavpuruniversity.in}}

\address{Gaston M. N'Guerekata \\ Department of Mathematics, Morgan State University\\ Baltimore, MD 21251, USA\\
\email{gaston.n'guerekata@morgan.edu}}

\CorrespondingAuthor{Gaston M. N'Guerekata}


\date{DD.MM.YYYY}                               



\thanks{Research of NB is supported by JU-RUSA 2.0} 

\begin{abstract}
       Recently, the dynamical behaviors of a fractional order three species food chain model was studied by Alidousti and Ghahfarokhi ({\it Nonlinear Dynamics, doi: org/10.1007/s11071-018-4663-6, 2018}). They proved both the local and global asymptotic stability of all equilibrium points except the interior one. This work extends their work and gives proof of both the local and global stability analysis of the interior equilibrium point. Numerical examples are also provided to substantiate the analytical findings.
\end{abstract}

\maketitle



\section{Introduction}

       Fractional calculus is a generalization of classical differential and integral calculus of integer order to arbitrary order. The notion of fractional derivative was first introduced by Leibnitz in 1695 and subsequently developed by Liouville, Heaviside, Caputo, Riemann. along with many others \cite{SET93}. Initially, fractional order derivatives and fractional order differential equations were treated as a topic of interest of pure mathematicians \cite{TET15}, but later on it found its own way of application in different fields of science and engineering mainly for two reasons. First, fractional order derivatives not only depends on the local conditions but also on the previous history of the function \cite{BET17}. Therefore, fractional derivatives became an efficient tool where consideration of memory or hereditary properties of the function is essential to represent the system, e.g., in case of biological systems. Secondly, fractional derivatives has an additional degree of freedom over its integer order counterpart due to the additional parameter that represents its order, and therefore more suitable for those systems having higher order dynamics and complex nonlinear phenomena \cite{TB84,SET11}.
       In the last two decades, fractional order calculus has been extensively used in several branches of science \& engineering and the number is huge. For brevity, we here mention only some review papers and books  \cite{MET11,GET10,AET12,D11,M06}. Fractional order models have also been used to understand the dynamics of interacting populations \cite{Ahmed07,RET13,CY14,MET17, LET16, V15, HET15, MET17a}.
       
       
       In recent past, Aziz-Alaoui \cite{Aziz02} studied the following three-dimension coupled nonlinear autonomous system of integer order differential equations to understand the underlying dynamics of food chain model:
       \begin{eqnarray}\label{Tritophic_model}\nonumber
       \frac{dX}{dT} & = & a_0 X - b_0 X^2 - \frac{v_0 XY}{d_0 + X},~~ X(0)\geq 0, \\
       \frac{dY}{dT}& = & -a_1 Y + \frac{v_1 XY}{d_1 + X} - \frac{v_2 YZ}{d_2 + Y},~~ Y(0)\geq 0, \\
       \frac{dZ}{dT}& = & c_3 Z^2 - \frac{v_3 Z^2}{d_3 + Y},~~ Z(0)\geq 0, \nonumber
       \end{eqnarray}
       where $X, Y, Z$ are, respectively, the densities of prey, intermediate predator and top predator at any instant of at time $T$. All parameters are non-zero positive. For description of the model and system parameters, readers are referred to \cite{Aziz02}. \\
       With the transformations
       $$X = \frac{a_0}{b_0}x,~ Y = \frac{a_0^2}{b_0 v_0}y,~ Z = \frac{a_0^3}{b_0 v_0 v_2}z,~ T = \frac{t}{a_0}$$
       and $$a = \frac{b_0 d_0}{a_0},~ b = \frac{a_1}{a_0},~ c = \frac{v_1}{a_0},~ d = \frac{d_2 v_0 b_0}{a_0^2},~ p = \frac{c_3 a_0^2}{b_0 v_0 v_2},~ q = \frac{v_3}{v_2},~ r = \frac{d_3 v_0 b_0}{a_0^2},$$ the system (\ref{Tritophic_model}) takes the simplified form
       \begin{eqnarray}\label{Tritophic_1_model}\nonumber
       \frac{dx}{dt} & = & x(1-x)- \frac{xy}{x + a},~ x(0) = x_0\geq 0, \\
       \frac{dy}{dt}& = & \frac{cxy}{x + a} -by -  \frac{yz}{y + d},~ y(0) = y_0\geq 0, \\
       \frac{dz}{dt}& = & pz^2 - \frac{qz^2}{y +r},~ z(0) = z_0\geq 0. \nonumber
       \end{eqnarray}
       This system admits three biologically feasible boundary equilibrium points and one interior equilibrium point. Local and global stability criteria of these boundary equilibrium points are given in \cite{Aziz02} except the interior equilibrium point. It is numerically shown there that the system exhibits chaos through period doubling bifurcation. Considering the fractional derivative in caputo sense, Alidousti and Ghahfarokhi \cite{Alidousti18} extended the work of Aziz-Alaoui \cite{Aziz02} and analyzed the following fractional order tri-trophic model:
       \begin{eqnarray}\label{Tritophic fractional order model}\nonumber
       \frac{d^mX}{dT^m} & = & a_0 X - b_0 X^2 - \frac{v_0 XY}{d_0 + X},~~ X(0)\geq 0, \\
       \frac{d^mY}{dT^m}& = & -a_1 Y + \frac{v_1 XY}{d_1 + X} - \frac{v_2 YZ}{d_2 + Y},~~ Y(0)\geq 0, \\
       \frac{d^mZ}{dT^m}& = & c_3 Z^2 - \frac{v_3 Z^2}{d_3 + Y},~~ Z(0)\geq 0, \nonumber
       \end{eqnarray}
       where $m \in (0,1)$ is the order of the derivative. With the same transformations as before, the system (\ref{Tritophic fractional order model}) takes the following simplified form :
       \begin{eqnarray}\label{Tritophic fractional order model_2}\nonumber
       ^{c}_{0} D^{m}_{t}x & = & x(1-x)- \frac{xy}{x + a},~ x(0)\geq 0, \\
       ^{c}_{0} D^{m}_{t}y & =  & \frac{cxy}{x + a} -by -  \frac{yz}{y + d},~ y(0)\geq 0, \\
       ^{c}_{0} D^{m}_{t}z & =  & pz^2 - \frac{qz^2}{y +r},~ z(0)\geq 0, \nonumber
       \end{eqnarray}
       where $^{c}_{0} D^{m}_{t}$ is the Caputo fractional derivative with fractional order $m$ $(0< m \leq1)$. They have shown that the solutions of system (\ref{Tritophic fractional order model_2}) are positively invariant and uniformly bounded in $R^3_+$ under some restrictions. Local and global stability of three boundary equilibrium points of system (\ref{Tritophic fractional order model_2}) were also proved. Stability analysis of the coexistence (or interior) equilibrium point, however, was omitted as in the case of integer order system. Their simulation results using realistic parameter values showed that the fractional order system (\ref{Tritophic fractional order model}) exhibits rich dynamics, like chaos, when the value of $m$ is close to $1 ~(m = 0.97)$, but exhibits regular oscillations $(\mbox{for} ~ m = 0.9)$, or even stable behavior $(\mbox{for} ~ m = 0.88)$ as the value of $m$ becomes smaller. We here extend the works of Alidousti and Ghahfarokhi \cite{Alidousti18} and Aziz-Alaoui \cite{Aziz02} by proving the local and global stability criteria of the interior equilibrium point for both the integer and fractional order systems. Simulation results are also given to validate the analytical results.

\section{Mathematical results}
Alidousti and Ghahfarokhi \cite{Alidousti18} proved the following results regarding positivity and boundedness of the solutions of system (\ref{Tritophic fractional order model_2}).

\begin{theorem}[]  
	\label{thefirstone}
	\textit{If}
\begin{equation}\label{cond_1}
c + \frac{c}{4b} +r < \frac{q}{p}
\end{equation}
\textit{ and $A$ be the set defined by
	$$A = \bigg\{(x,y,z) \in R^3_+ : 0\leq x \leq1, 0 \leq x+ \frac{y}{c} \leq 1 + \frac{1}{4b}, 0 \leq x + \frac{y}{c} + \alpha z \leq 1+ \frac{1}{4b} + \frac{M}{b} \bigg\},$$\\
	where $$\alpha = \frac{1}{b^2 (c +\frac{c}{4b} +r)}, ~M = \frac{1}{4(q-(c +\frac{c}{4b} +r)p)},$$ then \\
	$(i)$ $A$ is positively invariant, \\
	$(ii)$ all non negative solutions of system (\ref{Tritophic fractional order model_2}) initiating in $R^3_+$ are uniformly bounded in time and they enter the attracting set $A$.}\\
\end{theorem}       
       
\subsection{Existence and stability of equilibria}
\noindent The system (\ref{Tritophic fractional order model_2}) has four biologically feasible equilibrium points. The trivial equilibrium $E_{0} = (0, 0, 0)$ and the axial equilibrium  $E_{1} = (1,0,0)$ always exist. The planner equilibrium point $E_2 = (\theta, (1-\theta)(a +\theta),0)$ exists if $\theta<1$, where $\theta = \frac{ab}{c-b}$; or in other word $c> b(1+a)$. There exists a unique interior equilibrium point $E^* = (x^*, y^*, z^*) $ of the system (\ref{Tritophic fractional order model_2}), where the equilibrium population densities are given by
\begin{eqnarray}\label{Equilibrium relation}
x^{*} & = & \frac{(1-a)}{2} + \sqrt{\bigg(\frac{1+a}{2}\bigg)^2 - y^*}, ~~~y^{*} = \frac{q}{p} - r, ~z^* = (-b + \frac{cx^*}{a +x^*})(y^* + d).
\end{eqnarray}
The positivity condition of $E^*$ are $$a_0 > max\{b_0 d_0,~ 2\sqrt{b_0 v_0\bigg(\frac{v_3}{c_3} - d_3\bigg)} - b_0 d_0,~  \frac{b_0 d_0 a_1}{(v_1 - a_1)} + \frac{v_0}{d_0 v_1} \bigg(\frac{v_3}{c_3} -d_3\bigg)(v_1 - a_1) \},$$ where $v_3 > d_3 c_3$ and $v_1 > a_1$. Local and global stability results for the equilibrium points $E_0, E_1$ and $E_2$ are given in \cite{Alidousti18}. In the following, we give local and global stability results of $E^*$ only.

\section{\textbf{Main results}}

\noindent For local stability of the interior equilibrium $E^*$, we compute the Jacobian matrix of system (\ref{Tritophic fractional order model_2}) at $E^* = (x^*, y^*, z^*)$ as
\begin{equation}
\mathbf{J(E^{*})} = \begin{pmatrix}
\frac{x^*}{a + x^*} (1 - a - 2x^*) & \frac{-x^*}{a + x^*} & 0 \\ \frac{ac (1 - x^*)}{a + x^*} & \frac{y^*z^*}{(y^* + d)^2} & \frac{-y^*}{y^* +d}  \\ 0 & \frac{p(z^*)^2}{y^* + r} & 0
\end{pmatrix}.
\end{equation}
The eigenvalues are the roots of the cubic equation
\begin{equation}\label{cubic equation}
F(\xi) =  0,
\end{equation}
where $F(\xi) = \xi^{3} + A_{1} \xi^{2} + A_{2} \xi + A_{3}$,\\
$
A_{1} = \frac{x^*}{a + x^*} (2x^* - a - 1) - \frac{y^*z^*}{(y^* + d)^2},~ \\
A_{2} = \frac{py^*(z^*)^2}{(y^* + r)(y^* + d)} + \frac{x^* y^* z^*}{(a + x^*)(y^* + d)^2} (1 - a - 2x^*) +  \frac{acx^* (1 - x^*)}{(a + x^*)^2},\\
A_{3} = \frac{x^*}{a + x^*} (2x^* - a - 1)\frac{py^*(z^*)^2}{(y^* + r)(y^* + d)}.$\\

\noindent The equilibrium $E^*$ is said to be locally asymptotically stable if all eigenvalues of (\ref{cubic equation}) satisfy  $\mid arg(\xi_{i})\mid >\frac{m \pi}{2}, \forall m \in (0,1]$, $i = 1,2,3$. One can then determine the stability of $E^*$ by noting the signs of the coefficients $A_i$ and discriminant $D(F)$ of the cubic polynomial $F(\xi)$ \cite{Ahmed07,Ahmed06}. The discriminant $D(F)$ of the cubic polynomial $F(\xi)$ is
\[
\mathbf{D(F)} = - \begin{vmatrix}
1 & A_{1} & A_{2} & A_{3} & 0 \\ 0 & 1 & A_{1} & A_{2} & A_{3} \\ 3 & 2A_{1} & A_{2} & 0 & 0 \\ 0 & 3 & 2A_{1} & A_{2} & 0 \\ 0 & 0 & 3 & 2A_{1} & A_{2}
\end{vmatrix}= 18A_{1}A_{2}A_{3} + (A_{1}A_{2})^2 - 4A_{3}A_{1}^{3} - 4A_{2}^{3} - 27A_{3}^{2}.
\]
Then the following theorem regarding local asymptotic stability of $E^*$ of the system (\ref{Tritophic fractional order model_2}) is true \cite{Ahmed07,Ahmed06,Bairagi17}.\\

\begin{theorem}
	\label{thesecondone}
	\begin{itemize}
	\item[(i)] If $D(F) > 0$,  $ A_{1}>0$, $A_{3}>0$ and $A_{1}A_{2}- A_{3}>0$ then the interior equilibrium $E^{*}$ is locally asymptotically stable for all $m \in (0, 1]$.
	
	\item[(ii)] If $D(F) < 0$, $A_{1}\geq 0$, $A_{2}\geq 0$, $A_{3} > 0$ and $0< m < \frac{2}{3}$ then the interior equilibrium $E^{*}$ is locally asymptotically stable.
	
	\item[(iii)] If $D(F) < 0$, $A_{1} < 0$, $A_{2} < 0$ and $m > \frac{2}{3}$ then the interior equilibrium $E^{*}$ is unstable.
	\item[(iv)] If $D(F) < 0$, $A_{1} > 0$, $A_{2} > 0$, $A_{1}A_{2} = A_{3}$ and $0< m < 1$ then the interior equilibrium $E^{*}$ is locally asymptotically stable.
    \end{itemize}
\end{theorem}

\noindent To prove the global stability of $E^*$, we use the following Lemma \cite{V15} . \\

\begin{lemma}
	\label{one}
	 \textit{Let $x(t)\in \Re_{+}$ be a continuous and derivable function. Then for any time instant $t>t_{0}$
	\begin{equation}\nonumber
	^{c}_{t_{0}} D^{m}_{t}\bigg[x(t) - x^{*} - x^{*}ln\frac{x(t)}{x^{*}}\bigg] \leq \bigg(1-\frac{x^{*}}{x(t)}\bigg)~~{^{c}_{t_{0}}} D^{m}_{t}x(t), x^{*}\in \Re_{+}, \forall m\in(0,1].
	\end{equation}}
\end{lemma}

\begin{theorem}
	\label{thethirdone} \textit{The interior equilibrium $E^* = (x^*, y^*, z^*)$ of system (\ref{Tritophic fractional order model_2}) is globally asymptotically stable for any $m\in (0, 1]$ if} 
\begin{eqnarray}\label{cond_2} \nonumber
(i)&\frac{y^*}{a(a + x^*)} - 1 <0, \nonumber \\
(ii)&\frac{a +x^*}{ac}\bigg(\frac{z^*}{d(d + y^*)} - \frac{1}{2(c + \frac{c}{4b} + d)}\bigg) +   \frac{q}{2br\alpha} < 0, \nonumber \\
(iii)&\frac{q}{br\alpha} - \frac{a + x^*}{ac(c + \frac{c}{4b} + d)} < 0,\nonumber
\end{eqnarray}
where $$\alpha = \frac{1}{b^2 (c +\frac{c}{4b} +r)}>0.$$\\
\end{theorem}

\begin{proof} Let us consider the Lyapunov function
\begin{equation}\nonumber
V(x,y,z) = \bigg(x - x^* - x^*ln\frac{x}{x^*}\bigg) + \frac{a+x^*}{ac}\bigg(y - y^* - y^*ln\frac{y}{y^*}\bigg) + (y^* + r)\bigg(z - z^* - z^*ln\frac{z}{z^*}\bigg).
\end{equation}
It is easy to see that $V = 0$ only at $(x, y, z) = (x^*, y^*, z^*)$ and $V > 0$ whenever $(x, y, z) \neq (x^*, y^*, z^*)$. Considering the $m-th$ order fractional derivative of $V(x,y,z)$ along the solutions of (\ref{Tritophic fractional order model_2}), we have
\begin{eqnarray}
\begin{aligned}
^c_{0}D^{m}_{t}V(x,y,z) = & {^c_{0}}D^{m}_{t}\bigg(x - x^* - x^*ln\frac{x}{x^*}\bigg) + \frac{a+x^*}{ac} {^c_{0}}D^{m}_{t}\bigg(y - y^* - y^*ln\frac{y}{y^*}\bigg) \\
& + (y^* + r){^c_{0}}D^{m}_{t}\bigg(z - z^* - z^*ln\frac{z}{z^*}\bigg). \\
 \end{aligned}
\end{eqnarray}
Using Lemma \ref{one}, we have
\begin{eqnarray}\label{Globality_2}\nonumber
\begin{aligned}
^c_{0}D^{m}_{t}V(x,y,z) \leq & \frac{(x - x^*)}{x} {^c_{0}}D^{m}_{t}x(t) + \frac{a+x^*}{ac} \frac{(y - y^*)}{y} {^c_{0}}D^{m}_{t}y(t) \\
& + (y^* + r)\frac{(z - z^*)}{z} {^c_{0}}D^{m}_{t}z(t).
\end{aligned}
\end{eqnarray}
Following \cite{LET16,MET17a}, one can easily prove
\begin{eqnarray}\nonumber
\begin{aligned}
^c_{0}D^{m}_{t}V(x,y,z)
\leq & \bigg[\frac{y^*}{a(a+x^*)} - 1\bigg] (x - x^*)^{2} + \bigg[\frac{a +x^*}{ac}\bigg(\frac{z^*}{d(d + y^*)} - \frac{1}{2(c + \frac{c}{4b} + d)}\bigg) \\
& +   \frac{q}{2br\alpha}\bigg] (y-y^*)^2 + \frac{1}{2}\bigg[\frac{q}{br\alpha} - \frac{a + x^*}{ac(c + \frac{c}{4b} + d)}\bigg](z-z^*)^2\\
\leq & 0, \forall (x,y,z)\in \Re^{3}_{+}
\end{aligned}
\end{eqnarray}
if the following conditions hold:
\begin{eqnarray}\nonumber
\begin{aligned}
\frac{y^*}{a(a + x^*)}- 1 &< 0,\\
\frac{a +x^*}{ac}\bigg(\frac{z^*}{d(d + y^*)} - \frac{1}{2(c + \frac{c}{4b} + d)}\bigg) +   \frac{q}{2br\alpha} &< 0,\\
\frac{q}{br\alpha} - \frac{a + x^*}{ac(c + \frac{c}{4b} + d)} &< 0.
\end{aligned}
\end{eqnarray}
Here $^c_{0}D^{m}_{t}V(x,y,z) = 0$ implies that $(x,y,z) = (x^*, y^*, z^*)$. Therefore, the only invariant set on which $^c_{0}D^{m}_{t}V(x,y,z) = 0$ is the singleton set $\{E^*\}$. Then, using Lemma $4.6$ in \cite{HET15}, it follows that the interior equilibrium $E^*$ is global asymptotically stable for any $m \in (0,1]$ if conditions of Theorem \ref{thethirdone} are satisfied. Hence the theorem is proven.\\
\end{proof}

\begin{remark} This global stability result is independent of fractional order $m$ and it is also true for integer order $(m = 1)$.
	\end{remark}

\section{Numerical Simulations}
In this section, we perform extensive numerical computations of the fractional order system (\ref{Tritophic fractional order model}) for different fractional values of $m$ $(0 < m < 1)$ and also for $m=1$. We use Adams-type predictor corrector method (PECE) for the numerical solution of system (\ref{Tritophic fractional order model}). It is an effective method to give numerical solutions of both linear and nonlinear FODE \cite{Diethelm02, Diethelm04}. We first replace our system (\ref{Tritophic fractional order model}) by the following equivalent fractional integral equations:
\begin{eqnarray}\label{Tritophic fractional integral eqn} \nonumber
X(T) & = & X(0) + D^{-m}_{T} [a_0 X - b_0 X^2 - \frac{v_0 XY}{d_0 + X}], \nonumber\\
Y(T) & = & Y(0) + D^{-m}_{T} [-a_1 Y + \frac{v_1 XY}{d_1 + X} - \frac{v_2 YZ}{d_2 + Y}], \\
Z(T) & = & Z(0) + D^{-m}_{T} [c_3 Z^2 - \frac{v_3 Z^2}{d_3 + Y}]. \nonumber
\end{eqnarray}
and then apply the PECE (Predict, Evaluate, Correct, Evaluate) method.\\

Several examples are presented to illustrate the analytical results obtained in the previous section. To explore the effect of fractional order on the system dynamics, we varied $m$ in its range $0< m<1$. We also plotted the solutions for $m=1$, whenever necessary, to compare the solutions of fractional order system with that of integer order. It is to be mentioned that we first rescale all conditions and then verify different stability conditions of the system (\ref{Tritophic fractional order model}).\\

\noindent \textbf{Example 1:}
We considered the parameter values as $v_0 = 1.0$, $d_0 = d_1 = d_2 = 10.0$, $a_1 = 1.0$, $v_1 = 2.0$, $v_2 = 0.405$, $v_3 = 1.0$, $c_3 = 0.038$, $d_3 = 20.0$ and initial point $X(0) = 1.2, Y(0) = 1.2, Z(0) = 1.2$ from Aziz-Alaoui \cite{Aziz02} except $b_0 = 0.075$. Step size for all simulations is considered as $0.05$. This parameter set satisfies the positivity conditions of $E^*$, viz., $v_3>d_3c_3$, $v_1>a_1$ and $a_0 > max \{0.7500, 0.6265, 1.0658\}$. Thus we choose $a_0 = 1.2$ and compute $D(F) = 0.00018925 > 0$, $ A_{1} = 0.6007 >0$, $A_{3} =  0.0016>0$, $A_{1}A_{2}- A_{3} = 0.0465>0$. Therefore, following Theorem 2 (i), the interior equilibrium $E^{*} = (12.2081, 6.3158, 4.0056)$ of (\ref{Tritophic fractional order model}) is locally asymptotically stable for $0 < m \leq 1$. Fig. 1 represents the behavior of solutions of FDE system (\ref{Tritophic fractional order model}) for different values of $m$, depicting the stability of interior equilibrium point $E^{*}$. It is noticeable that solutions reach to equilibrium value more slowly as the value of $m$ becomes smaller.\\

\noindent \textbf{Example 2:} If we consider $b_0 = 0.06$, as in \cite{Aziz02}, leaving other parameter values unchanged, then $E^*$ exists if $a_0 > max \{0.6000, 0.6312, 0.9158\}$. Selecting $a_0 = 0.95$, we observe that the conditions of Theorem 2 (ii) are satisfied with $D(F) = -0.0100 <0$, $ A_{1} =0.4988>0$, $A_{2} =  0.1611 > 0, A_3 = 0.00028 > 0$. Therefore, the interior equilibrium point $E^{*} = (10.7638, 6.3158, 1.4819)$ of (\ref{Tritophic fractional order model}) is locally asymptotically stable for $0<m < \frac{2}{3}$ as shown in Fig. 2.\\

\begin{figure}[H]
	\centering
	\includegraphics[width=9in, height=2.5in]{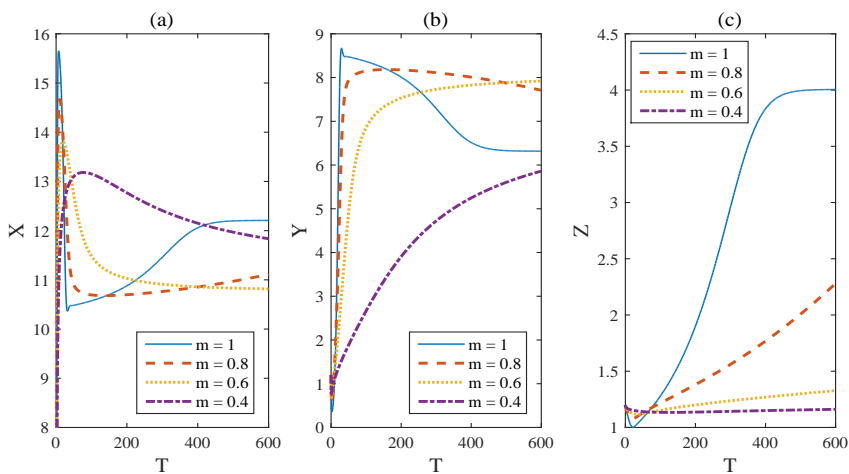}
	\caption{Asymptotically stable solutions of $X$, $Y$ and $Z$ populations with different fractional orders $ 0< m < 1$ and integer order $m = 1$. Here $b_0 = 0.075$, $v_0 = 1.0$, $d_0 = d_1 = d_2 = 10.0$, $a_1 = 1.0$, $v_1 = 2.0$, $v_2 = 0.405$, $v_3 = 1.0$, $c_3 = 0.038$, $d_3 = 20.0$ and $a_0 = 1.2$.}
\end{figure}

\begin{figure}[H]
	\centering
	\includegraphics[width=9in, height=2.5in]{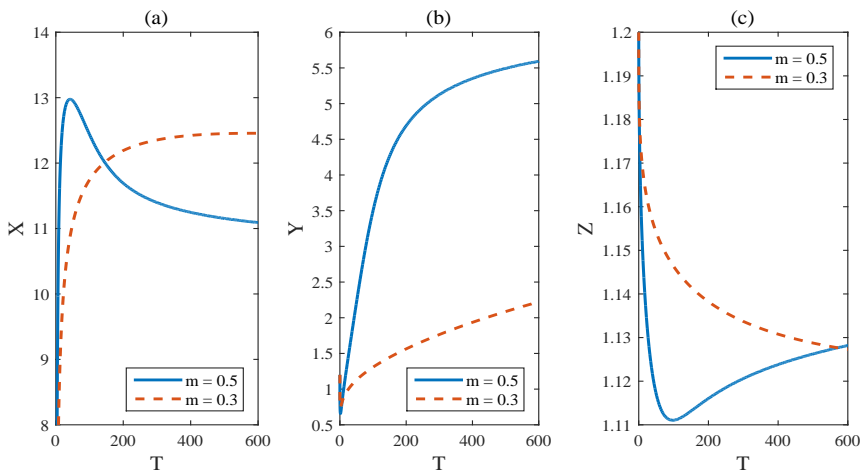}
	\caption{Asymptotically stable solutions of $X$, $Y$ and $Z$ populations with different fractional orders $ 0< m < \frac{2}{3}$. Parameters are as in Example 1 except $b_0 = 0.06$ and $a_0 = 0.95$.}
\end{figure}

\noindent \textbf{Example 3:} If we consider $v_1 = 10, v_2 = 2.5$, keeping other parameter values unchanged as in Example 1, then $E^*$ exists if $a_0 > max \{0.7500, 0.6265, 0.6518\}$. We then choose $a_0 = 1.5$ and verify that all the conditions of Theorem 2 (iii) are satisfied with $D(F) = -7.4129 <0$, $ A_{1} = -0.6171< 0$, $A_{2} =  -0.0335<0$. Therefore, the interior equilibrium point $E^{*} = (16.8655, 6.3158, 34.4443)$ of (\ref{Tritophic fractional order model}) is unstable for  $m > \frac{2}{3}$ (Fig. 3).
\begin{figure}[H]
	\centering
	\includegraphics[width=9in, height=3.5in]{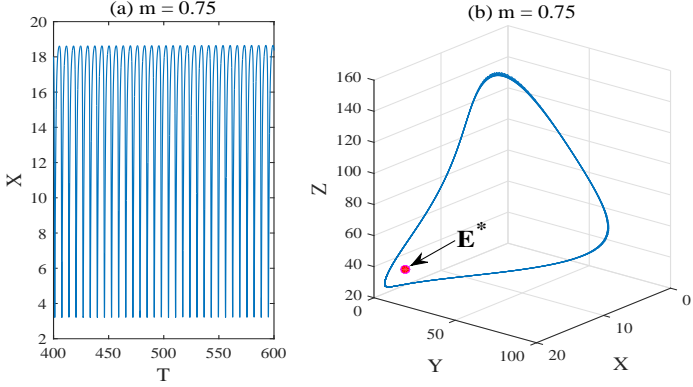}
	\vspace{-2cm}
	\caption{Unstable time evolution of $X$ population (Fig. 3a) and corresponding phase plane (Fig. 3b). Here $m = 0.75(>\frac{2}{3})$, $v_1 = 10, v_2 = 2.5, a_0 = 1.5$ and other parameters are as in Fig. $1$.}
\end{figure}

\noindent \textbf{Example 4:} To demonstrate the global stability of the interior equilibrium point $E^*$, we consider the parameter values $b_0 = 0.15$, $v_0 = 1.0$, $d_0 = d_1 = d_2 = 10.0$, $a_1 = 1.0$, $v_1 = 2.0$, $v_2 = 2.5$, $v_3 = 1.0$, $c_3 = 0.038$, $d_3 = 20.0$ and different initial points $(1.2, 1.2, 1.2)$, $(10.1,30.1,3)$, $(30,10,5)$, $(25,5,1)$, $(22,5,4)$, $(18,15,8)$, $(12,20,2)$, $(5,30,6)$. In this case, $E^*$ exists if $a_0 > max \{1.5000, 0.4467, 1.8158\}$ and so we consider $a_0 = 2.0$. With these parameter values, we verify that all conditions of Theorem $3$ are satisfied as $\frac{y^*}{a(a + x^*)} -1 = -0.8029 < 0$,$\frac{a +x^*}{ac}\bigg(\frac{z^*}{d(d + y^*)} - \frac{1}{2(c + \frac{c}{4b} + d)}\bigg) +   \frac{q}{2br\alpha}  = -0.2804 < 0$, $\frac{q}{br\alpha} - \frac{a + x^*}{ac(c + \frac{c}{4b} + d)}  = - 0.9242 < 0$, where $\alpha = \frac{1}{b^2 (c +\frac{c}{4b} +r)} = 2.1333>0$.  Fig. 4 demonstrates that solutions starting from different initial values converge to the equilibrium point $E^{*} = (11.3623, 6.3158, 0.4162)$ of (\ref{Tritophic fractional order model}) for different fractional orders, $m = 0.65, 0.75, 0.85$, and also for the integer order, $m = 1$, depicting the global stability of the interior equilibrium point for fractional order as well as integer order.

\begin{figure}[H]
	\centering
	\includegraphics[width=9in, height=3in]{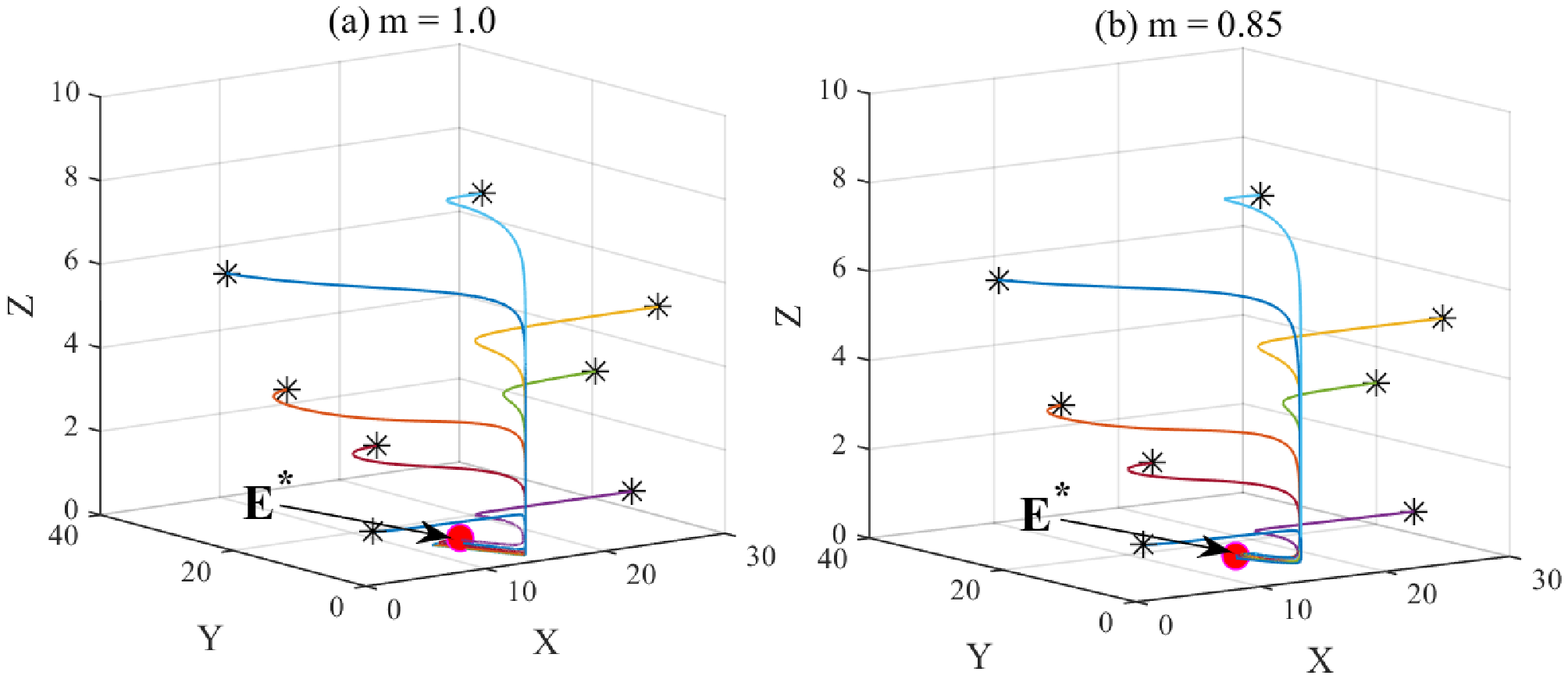}
\end{figure}
\begin{figure}[H]
	\vspace{-1.5cm}
	\centering
	\includegraphics[width=9in, height=3in]{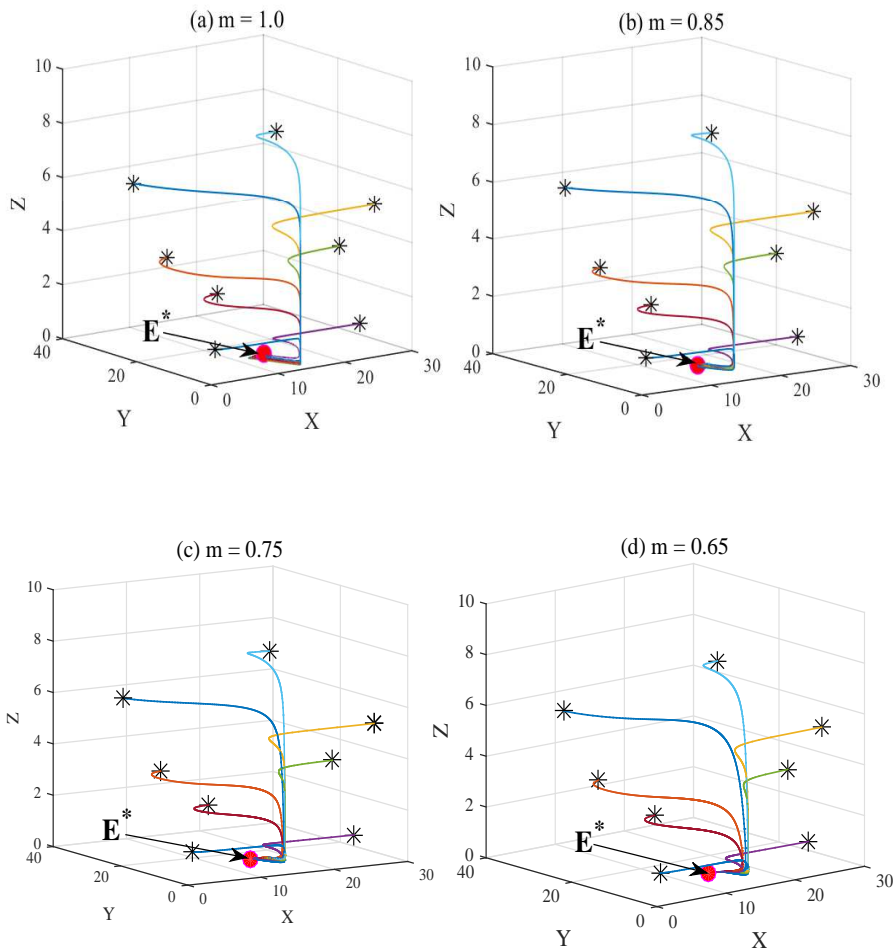}
	\vspace{-1.5cm}
	\caption{Trajectories with different initial values converge to the interior equilibrium point $E^*$ for different values of $m$, indicating global stability of the equilibrium $E^*$ when conditions of Theorem 3 are satisfied. All parameters are as in Fig. 1 except $a_0 = 2.0$, $b_0 = 0.15$ and $v_2 = 2.5$. Initial values are indicated with stars and equilibrium point is denoted by red circle.}
\end{figure}

\noindent \textbf{Example 5:} Here we consider the exact parameter set and initial value as in Alidousti and Ghahfarokhi \cite{Alidousti18} and reproduce their bifurcation diagrams (Figs. 5a and 5b) with respect to the same growth rate parameter of prey (here it is $a_0$) in the same range $[1.6, 2.1]$ for the orders $m=1$ and $m=0.97$. As shown in \cite{Aziz02,Alidousti18}, the system (\ref{Tritophic fractional order model}) exhibits complex chaotic dynamics through period-doubling bifurcation. The first period-doubling bifurcation occurs at $a_0 \approx 1.66 $ for the integer order $m = 1.0$~ (Fig. 5a) and it occurs (Fig. 5b) at $a_0 \approx 1.69$ for the fractional order $m = 0.97$ \cite{Alidousti18}. If we consider our global parameter set of Example $4$ with the same initial values as in \cite{Alidousti18} and draw similar bifurcations (Figs. 5c and 5d) then no bifurcation and complex dynamics is observed because our equilibrium point is globally stable for both the integer and fractional orders. 

\begin{figure}[H]
	\hspace{-.5cm}
	\includegraphics[width=5in, height=2.2in]{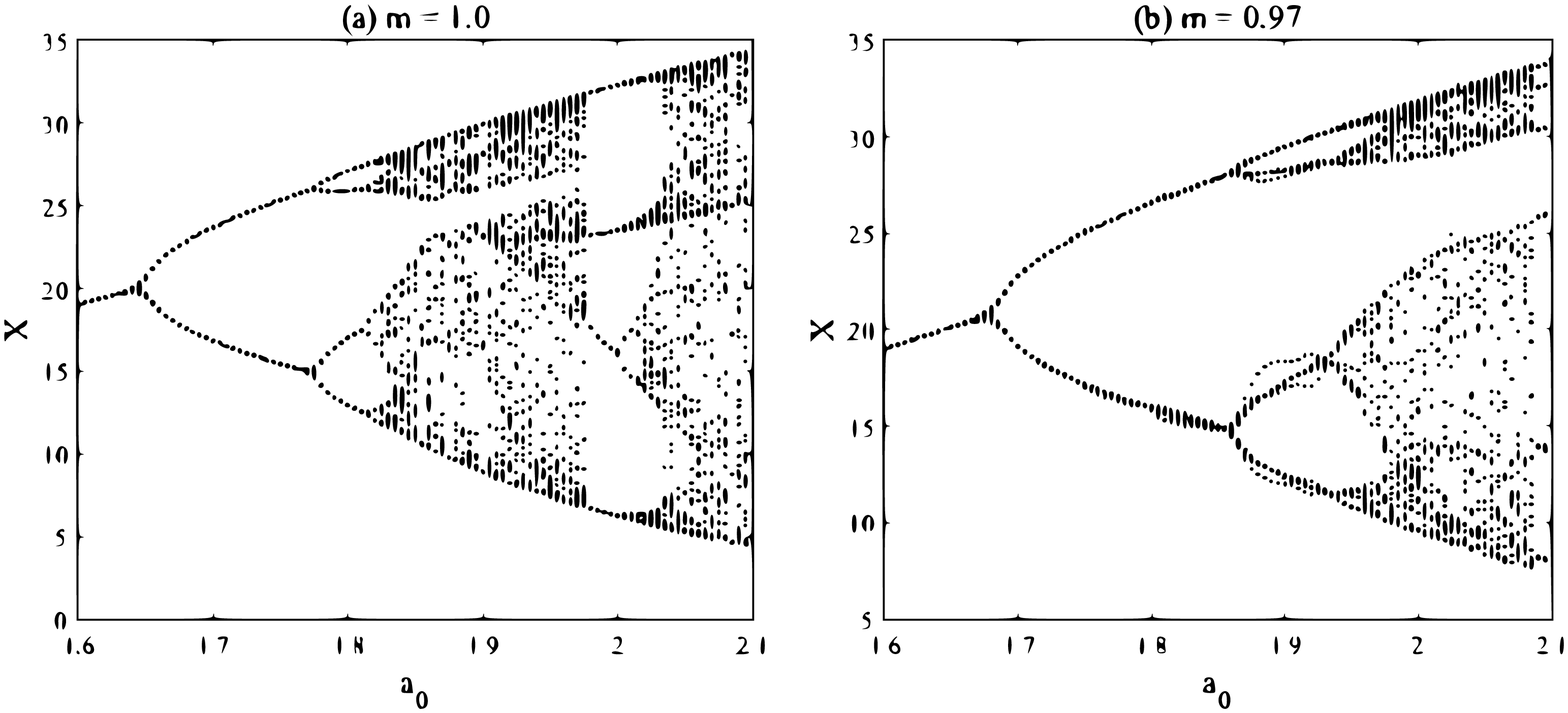}
\end{figure}
\begin{figure}[H]
	\vspace{1cm}
	\hspace{-.6cm}
	\includegraphics[width=5in, height=2.2in]{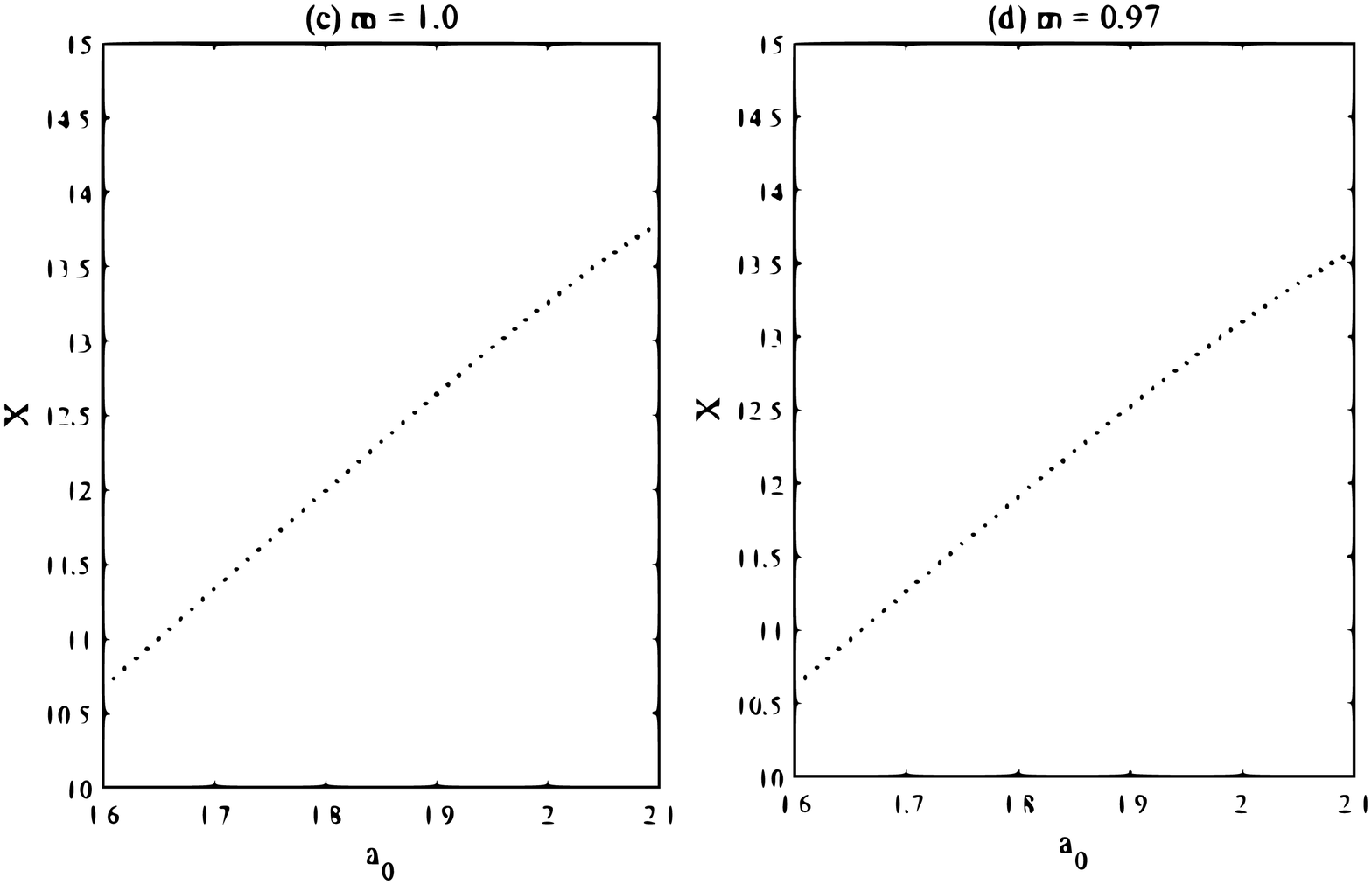}
	\caption{Bifurcation diagrams of $X$ population of system (\ref{Tritophic fractional order model}) as shown in \cite{Aziz02,Alidousti18} in the range $[1.6, 2.1]$ with $a_0$ as the bifurcation parameter. System becomes unstable for integer order $m = 1.0$ at $a_0 \approx 1.66$ (Fig. 5a)) and then becomes chaotic for higher value of $a_0$. Similar dynamics is also observed with first bifurcation at $a_0 \approx 1.69$ (Fig. 5b) for fractional order $m = 0.97$. All parameters are as in Example 2. Similar bifurcation diagrams (Figs. 5c, 5d)  with the global parameter set of Example 4 show that there is no bifurcation with respect to the parameter $a_0$, indicating stability of the system.}
\end{figure}

\section{Summary}
In this paper, we extended the works of Alidousti and Ghahfarokhi \cite{Alidousti18} on fractional-order three-species food chain model and Aziz-Alaoui \cite{Aziz02} on corresponding integer order model by giving proof of local and global stability of the interior equilibrium point. For local stability we used Routh-Hurwitz criterion for fractional order differential equations. We defined suitable Lyapunov function to prove that the interior equilibrium is globally asymptotically stable if the system parameters satisfy some conditions. In such a case, the system does not show any complicated dynamics like chaos as shown in the earlier studies \cite{Aziz02,Alidousti18}, indicating its global stability. This is more reinforced by the fact that solutions initiating from biologically feasible arbitrary initial points converge to the interior equilibrium point.\\



\end{document}